\documentclass[11pt,onecolumn]{article}
\usepackage{amscd}
\usepackage{enumitem}
\usepackage{times}
\usepackage{amsmath}
\usepackage{amssymb}
\usepackage{xspace}
\usepackage{theorem}
\usepackage{graphicx}
\usepackage{ifpdf}
\usepackage{url,hyperref}
\usepackage{latexsym}
\usepackage{euscript}
\usepackage{xspace}
\usepackage[all]{xy}
\usepackage{color}
\usepackage{makeidx}
\usepackage{tikz}
\usepackage{enumitem}
\usepackage{graphics}
\long\def\remove#1{}
\newtheorem{theorem}{Theorem}[section] 

\newtheorem{obs}[theorem]{Observation}

\newtheorem{definition}[theorem]{Definition}
\newtheorem{proposition}[theorem]{Proposition}
\newenvironment{proof}{{\em Proof:}}{\hfill{\hfill\rule{2mm}{2mm}}}

\newcommand {\mm}[1] {\ifmmode{#1}\else{\mbox{\(#1\)}}\fi}

\newcommand{\img}{\mathrm img}

\newcommand{\rank}                {\mm {\rm rank}}

\newcommand{\cancel}[1]

\begin{document}

\title {Dynamics, Cohomology and Topology}

\author{
Dan Burghelea  \thanks{
Department of Mathematics,
The Ohio State University, Columbus, OH 43210,USA.
Email: {\tt burghele@math.ohio-state.edu}}
}
\date{}

\maketitle

\begin{abstract}
For a smooth Morse-Smale vector field  with Lyapunov constraints (Lyapunov function)
one shows  how and why  the non-triviality of the  cohomology,  as concluded from its additive structure,  detects  rest points   
and the multiplicative structure of the cohomology detects instantons (trajectories between rest points). The same remains true for  Lyapunov closed one form , a more general Lyapunov constraint, but in this presentation this fact is discussed only informally. These observations are  based on the smooth {\it manifold with corner structures} of the stable/unstable sets and of the set of trajectories  of such vector fields. 

\vskip .1in
\end{abstract} 

\setcounter{tocdepth}{1}
\tableofcontents

\section {Introduction} \label {S1}
 This talk \footnote {Lecture given at IMAR Bucharest, November 2023} is about a large class of dynamics,  M-S (Morse--Smale) dynamics or  M-S  vector fields,  whose basic elements are {\it rest points, instantons and periodic trajectories}. I  intend to explain 
 how the cohomology permits to conclude  existence 
 of these elements  (and to some extent derive information about their amount) under the hypotheses of the existence of a Lyapunov function or of a Lyapunov closed  one form. Very many  dynamics in physics   satisfy such Lyapunov  constraints. 
Note that M-S vector fields are   generic, as explained  in section \ref{S3}.
For reasons of simplicity I will suppose that  the underlying manifold of the M-S dynamics is {\it closed } but this hypothesis is  not necessary and  can be removed with appropriate modifications in definitions. 
 
The detection of the rest points, in the case of a Lyapunov function resp.of a Lyapunov closed differential one form of cohomology class $\xi,$ is derived  from the additive structure  of the cohomology $H^\ast (M)$ resp. of the Novikov cohomology $H_N^\ast(M;\xi);$ i.e. from  the dimensions of  components  of 
the graded $\kappa\ \rm{resp.}\  \kappa(\xi)-$ vector space structure of these cohomologies\footnote {here the cohomology $H^\ast(M)$ is a $\kappa-$graded vector space while Novikov cohomology is 
a $\kappa \langle \xi\rangle-$vector spaces, with $\kappa\langle \xi\rangle$ the Novikov  field, as discussed in section \ref {S2} } ,  
cf.Theorem \ref{T41}. This is a well known result in topology and treated by what is usually referred to as  elementary Morse theory resp. Morse--Novikov theory.
 
 The detection of instantons  is a consequence of the  multiplicative structure of the cohomology $H^\ast(M)$  (i.e. the graded algebra structure of the cohomology $H^\ast (M)$) resp. the left $H^\ast(M)-$module structure of $H^\ast_N (M;\xi),$ cf. Theorem \ref{T41}. 
 This last fact was  not known to me until very recently, and apparently not explicit in literature but, I understand it was known to S. Donaldson  as suggested by \cite {AB}.

 The purpose of this talk is to explain these facts based on the {\it manifold with corners}-structures and the explicit description of the corners of the unstable sets and of the spaces of trajectories  of a
 M-S vector field in the presence of Lyapunov constraints. In this paper we will treat in details only the case of Lyapunov function. The details in the case of Lyapunov closed differential one form is based 
 essentially on similar arguments  and will be presented  elsewhere.

\section {Cohomology}\label {S2}

 For a space $X$ and a field $\kappa$  the collection of vector spaces $H^r(X):= H^r(X;\kappa), r= 0,1, \cdots,$ \footnote{when the field $\kappa$ is understood from the context it  will be  deleted from notations} provides a graded $\kappa-$vector space.  Equipped with the products $H^r(X)\otimes H^p (X) \xrightarrow \cup  H^{r+p}(X),$ called cup-products \footnote {the cup-product is the linear map induced in cohomology by the diagonal map $ \Delta: X\to X\times X$ }, this graded vector space becomes   a commutative graded algebra.

If $\xi\in H^1(X;\mathbb R)= Hom (H_1(M,\mathbb Z), \mathbb R) $, since $M$ is a closed manifold, then  $\img \  {\xi}$ is a finite rank abelian group $\Gamma$  embedded in $\mathbb R,$ hence isomorphic to $\mathbb Z^p$ for some integer $p.$
This  group defines the extension field $ \kappa \langle \xi\rangle$ 
of the field $\kappa,$ known as the Novikov field  cf. \cite{F}, and the $\Gamma-$principal  covering $\pi:\tilde M\to M$  
which makes   $H^\ast (\tilde M)$ a  left $\kappa[\Gamma]-$ module for the left action,  
$x\cup u:= \pi^\ast(x) \cup u,$ $x\in H^\ast(M), u\in H^\ast (\tilde M).
$ Here $\cup$ denotes the cup-product in $\tilde M$.  
The Novikov cohomology $H^r_N(M;\xi):=H^r(\tilde M)\otimes _{\kappa[ \Gamma]} \kappa \langle \xi \rangle $ is a $\kappa \langle \xi\rangle-$vector space 
after the right-side tensor product with $\kappa \langle \xi \rangle, $
equipped  with a structure of left  $H^\ast(M)-$module, derived from the left $H^\ast(M)-$module structure of $H^\ast(\tilde M).$
 This module structure is provided by the linear map $\tilde \Delta: \tilde M\to M\times \tilde M, \tilde \Delta (x)= (\pi(x), x)$ which induces 
$H^\ast (M) \otimes_\kappa H^\ast (\tilde M)\otimes_{\kappa[\Gamma]} \kappa \langle \xi\rangle \to H^\ast(\tilde M)\otimes_{\kappa[\Gamma]} \kappa \langle \xi\rangle .$  
 If $\omega\in \Omega^1(M), d\omega=0,$ representing the cohomology class $\xi,$ then $\pi^\ast \omega= d f$ with $ f:\tilde M\to \mathbb R$ called {\it lift of $\omega$}, unique up to an additive constant. For more details consult \cite {F}, \cite {BH}. 

\section {Dynamics}\label {S3}

A smooth dynamics on a smooth manifold $M^n$ is a  one parameter group of diffeomorphisms $\varphi : \mathbb R\times M\to M,$ equivalently a  complete smooth vector field $X$ on $M.$

The elements of the dynamics $X$ are  the set of rest points $\mathcal X:= \{ x\in M\mid X(x)=0\}$
and the set of its trajectories, i.e. smooth maps $\gamma : \mathbb R\to M$ with $d\gamma(t)/ dt= X(\gamma(t)).$ One writes $\gamma_p(t), p\in M,$ for the trajectory with $\gamma_p(0)=p, $ and  $\gamma^\pm_p$ for the  restrictions of $\gamma_p$  to $(-\infty ,0]$ resp.   $[0,\infty).$

 Among these trajectories one specifies :
 \begin {itemize}
 \item  the {\bf instantons} $\mathcal T:= \{\gamma \mid \lim_{t\to -\infty / \infty} \gamma (x)= x /y \mid x,y\in \mathcal X \},$  trajectories between rest points, and
 \item the {\bf periodic trajectories} $\mathcal P: = \{\gamma \mid  \gamma(t+T)=\gamma(t) \rm \ {for\ some} \ T\}.$ 
 \end{itemize} as trajectories of interest.

 Note that one also has:
 \begin{itemize} 
 \item for any $x\in \mathcal X$  two relevant subsets  of $M,$ the {\bf unstable} resp. {\bf  stable set} 
 $$W^\mp_x: \{ y\in M  \mid \lim_{t \to - \infty/+\infty} \gamma_y(t)=x\},$$    
\item for any $x,y\in \mathcal X$  the subset  $\mathcal M(x,y)\subset M,$ $$ \mathcal M(x,y):= W^-_x\cap W^+_y,$$  consisting of the points located on the instantons from  $x$ to $y,$ 
and 
\item when $x\ne y,$ 
the quotient space, $$\mathcal T(x,y): = \mathcal M(x,y) / \mathbb R,$$ with respect to the free action  of $\mathbb R$ on $\mathcal M(x,y)$ induced by $\varphi,$ and
with $\pi:\mathcal M(x,y)\to \mathcal T(x,y)$ denoting the quotient map. 
\end{itemize}

Let   $$\pi\times f:\mathcal M(x,y)\to \mathcal T(x,y)\times (f(y), f(x)) \simeq \mathcal T(x,y)\times (0,1),$$
be the canonical identification  of $\mathcal M(x,y)$ to $\mathcal T(x,y)\times (f(y), f(x)),$ ultimately to $\mathcal T(x,y)\times (0, 1),$ once an identification  of the interval $(f(y), f(x))$ to $(0,1)$ is chosen.

For  $x\ne y$ one  writes $x>y$ if $\mathcal T(x,y)\ne \emptyset.$
 
\vskip .2in
{\bf M-S (Morse--Smale) vector fields}, cf. \cite {KH}, \cite{I}
\vskip .1in
In this paper we call {\it M-S vector field}, or {\it M-S dynamics,}  a complete smooth vector field \footnote {when M is closed any vector field in complete} $X$ on a smooth manifold $M$ which satisfies the following properties.

\begin{enumerate}
\item  The rest points $\mathcal X=\{x\in M\mid X(x)=0\}$ are  hyperbolic \footnote {the linearization of the vector field $X$ at any rest point is invertible, cf. \cite{KH}}, hence isolated and  with a well defined index   $i(x)\in \{0,1,\cdots \dim M\}$ and with the sets $W^-_x$ resp. $W^+_x$ manifolds diffeomorphic to $\mathbb R^k$ resp. $\mathbb R^{n-k},$ $k= i(x).$ 
One has  $\mathcal X= \sqcup \mathcal X_k,$  $$\mathcal X_k=\{ x\in \mathcal X\mid  i(x)=k \}$$ and one denotes by $$i^\pm_x : W^\pm _x\to M$$  the corresponding injective immersion 
\footnote {which exists by Hadamard--Perron theorem, cf \cite{KH}}.
Note that If for any  $x\in \mathcal X$ there exists coordinates $(t_1, t_2, \cdots t_n) $ in a neighborhood of $x$, s.t. $$X(t_1, \cdots t_n) = -\sum _{1\leq i \leq k} t_i\partial / \partial t_i + \sum _{k+1\leq i \leq n} t_i\partial / \partial t_i$$
then  this property is satisfied and $i(x)= k.$ 
\item  For each  periodic trajectories $\gamma\in \mathcal P$ the vector field $X$ is hyperbolic in normal directions to $\gamma,$
which implies that  the trajectory $\gamma$ is isolated and has a well defined  index $i(\gamma)\in\{0,1, \cdots, \dim M-1\}.$
This is indeed the case and $i(\gamma)=k,$  if for any $x\in \gamma$ there exists coordinates $(t_1, t_2, \cdots t_n)$ in a neighborhood of $x$ s.t. $$X(t_1, \cdots t_n) = -\sum _{2\leq i \leq k+1} t_i\partial / \partial t_i + \sum _{k+2\leq i \leq n} t_i\partial / \partial t_i.$$
\item For any $x,y\in \mathcal X$ the maps $i^-_x$ and $i^+_y$ are transversal, equivalently the unstable manifold $W^-_x$ and the stable manifold $W^+_y$ are transversal,
which implies that the set  $\mathcal M(x,y)$ is a smooth manifold of dimension $i(x)- i(y)$ which is stably parallelizable, as explained  in section \ref {S7}. This implies that for $x\ne y$ the set $\mathcal T(x,y)$  is also a smooth stably parallelizable  manifold and has dimension $i(x,y)-1.$

One writes 
$$\mathcal T(r+p, p):= \sqcup_{\{x\in \mathcal X_{r+p},y\in \mathcal X_r\}} \mathcal T(x,y)$$
$$\mathcal T(r):= \sqcup_p \mathcal T(r+p,r)$$
$$\mathcal T=\sqcup_r \mathcal T(r)$$  where 
and one denotes by 
 $$i_{x,y} :\mathcal M(x,y)\to M$$ the corresponding  injective smooth immersion.
\end{enumerate}
Usually there are  more transversality conditions  assumed in item 3.  in order to qualify the dynamics $X$ to be  M-S, cf. \cite{KH} , \cite{I} or \cite{P};  one requires that the stable and the  unstable sets for  both rest points and closed trajectories,  which in view of item 1. and item 2. above are all manifolds, be  transversal.  For the purpose of this paper item 3. as formulated  suffices and we stick with this shorter list  of requirements for the concept of M-S vector field.

The following weaker version of Kupka-Smale theorem hods true.

\begin{theorem} (Kupka-Smale) \

The set of M-S vector fields  on a closed manifold $M$ is generic in the set of all smooth vector fields w.r. to any  $C^r-$topology, $r\geq 1.$ 

More precisely, 
for any smooth vector field $X$ and any $\epsilon>0$ one can find M-S vector fields  $\epsilon-$closed to $X$, and for any smooth M-S vector field $X$ there exists $\epsilon >0$ s.t. any smooth vector field $X',$ $\epsilon-$closed to $X,$ remains a M-S vector field.
Moreover, there exists homeomorphism  $\varphi: M\to M$ 
 which intertwines  the rest points, the instantons  and the periodic trajectories of $X$ with those of $X'$ keeping their indexes. 
\end{theorem}
\vskip .1in
{\bf Lyapunov constraints}
\vskip .1in
A smooth function $f:M\to \mathbb R$ is called {\it Lyapunov} for the vector field $X$  if $df (X) (x) \ne 0$ iff $x\notin \mathcal X$  and $df (X) (x) < 0,$ 
and 

a closed differential one form $\omega \in\Omega^1(M),$ $d\omega=0$ is called {\it Lyapunov} for the vector field $X$ provided $\omega(X) (x) \ne 0$ iff $x\notin \mathcal X$ and $\omega(X)(x) <0.$
\vskip .1in
Note that :

 --the  vector fields  which admit  Lyapunov function do not have closed trajectories (,hence the requirement 2 is not applicable), 

-- the vector fields  which admit  Lyapunov closed one form,  equivalently whose trajectories minimize an {\it action} defined locally by a smooth function,  can have closed trajectories. Concluding the existence  of and counting the periodic trajectories  is an interesting problem in dynamics  cf. \cite{BH2}.  
Many dynamics of physical interest are mathematically described by such vector fields.   

\section {Results} 
\label {S4}

\begin{theorem} \label {T41}
Suppose $M^n$ is a closed smooth manifold and  $X$ is an M-S vector field.
\begin{enumerate}
\item  If $X$ admits  a Lyapunov function, then 
\begin{enumerate} 
\item $\mathcal P=\emptyset,$  $\sharp \pi_0(\mathcal T) <\infty$ \footnote {$\pi_0 (\mathcal T)$ denotes the set of connected components of the space $\mathcal T$} and $\sharp \mathcal X< \infty,$
\item for any field $\kappa$ $\sharp \mathcal X_r  \geq  \dim H^r(M) $  and $ (-1)^k\sum_{r\leq k} \sharp \mathcal X_r\geq (-1)^k\sum_{r\leq k} \dim H^r(M),$ with $\sharp \cdots$ denoting the cardinality of the set $\cdots,$
\item  if the cup-product  $H^r(M) \cup H^p(M) \to H^{r+p}(M)$  is not trivial  then  $\mathcal T(r-1+p,p) \ne  \emptyset.$
\end{enumerate}

\item If $X$ admits a  Lyapunov closed  differential one form $\omega\in \Omega^1(M)$  of cohomology class $\xi= [\omega]$  then 
\begin{enumerate}
\item $\sharp \mathcal X_r  <\infty$ 
and for any $x,y\in \mathcal X$  $\pi_0(\mathcal T(x,y)) <\infty,$ 
\item   For any field $\kappa,$ $\sharp \mathcal X_r \geq \dim H^r_N(M, \xi) $ and  $ (-1)^k\sum_{r\leq k} \sharp \mathcal X_r\geq (-1)^k\sum_{r\leq k} \dim H^r_N(M,\xi),$
\item  If the multiplication $H^r(M) \otimes H^p_N(M,\xi) \to H^{r+p}_N(M;\xi)$ is nontrivial then $\mathcal T(r-1 +p, p)\ne \emptyset.$
\end{enumerate}  
\end{enumerate}
\end{theorem}
Items a. and b. in  parts 1 and part 2 are well known and represent statements  referred to as elementary Morse resp. Morse-Novikov theory. In this paper only part 1 will be argued in details, part 2 will be only superficially discussed.
\vskip .2in

\section {Differential topology- manifolds with corners} \label{S5}
Denote by $\mathbb R^n_+:= [0,\infty)^n$ with  $\mathbb R^n_+ (k)$ the subset of $\mathbb R^n_+$ with exactly $k$ coordinates equal to $0.$

\begin{definition}\

\begin{enumerate}
\item  An $\mathbb R^n_+ -$ manifold $W$ is a space equipped with a sheaf of continuous  functions locally isomorphic to $\mathbb R^n_+$  equipped with the sheaf of smooth functions.

Let  $W(k)$ denotes  the subspace of points corresponding to $\mathbb R^n_+ (k)$ and observe that 
$W(k)$ is a smooth $(n-k)$ dimensional manifold referred to as the $k-$corner of $W.$
 \item 
 An $n-$dimensional manifold with corner $(W, W(k))$ is an $\mathbb R^n_+ -$ manifold such that the topological closure  of each connected component of $W (k)$ is an $\mathbb R^{n-k}_+ -$ manifold.
\end{enumerate}
 \end{definition}

 The product of two manifolds with corners $(W_1, W_1(k))$ and $(W_2, W_2(k))$ is a manifold with corners $(W, W(k))$ with $W= W_1\times W_2$ and $ W(k)$ given by 
 \begin{equation} \label {E2}
 W(k)= \sqcup_{0\leq r\leq k} W_1(r) \times W_2(k-r).
\end{equation} 

 The extension of the differential calculus  from manifolds  and manifolds with boundary to manifolds with corners, like transversality, the integration theory for  differential forms   including  Stokes theorem are straightforward.
In particular one has:

 \begin{obs} \label{O42}\
 
If $(W, W(k))$ is a compact oriented (smooth) manifold with corners  of dimension $n$ and $\omega$ is a differential form of degree $n-1$ then 
$\int _{W(1)} \omega $ is convergent and 
\begin{equation} \label{EE2}
 \int_{W(1)} \omega = \int _W  d\omega.
\end{equation}
\end{obs}
  
  If $f: W\to M^p$ is a smooth map from a manifold with corners $(W, W(k))$ to a manifold $M^p$ and $N^{p-r} \subset M^p$ is  the image by an injective immersion of a manifold, $l: N\to M,$  one says that 
  $f$ is transversal to $N$ or to $l$  and one writes $f\pitchfork N$ or $f\pitchfork l,$
if the restriction $f|_{W(k)}$ of $f$ to $W(k)$ remains transversal to $N$ or $l$ for any $k.$ If this is the case one has:
  
\begin{obs}\label {O43}\

 If $f\pitchfork N$then  $f^{-1}(N)$  is a manifold with corners whose $k-$corner  is  $(f|_{W(k)})^{-1}(N).$ 
\end{obs}

For a  manifold with corners $(W, W(k))$ one denotes by  $\partial W=\sqcup _{k\geq 1} W(k)$  the {\it boundary} of $W;$ actually $(W, \partial W)$ is a topological manifold with boundary which is smoothable.  Precisely, it admits a rounding-corner homeomorphisms to a smooth manifold with boundary 
 A {\bf rounding-corners} homeomorphism is a smooth homeomorphism $h: (W, \partial W)\to (N,\partial N),$ with $(N,\partial N)$ a  smooth manifold with boundary which for any $k$ restricts to a  diffeomorphism onto the image and the stratification of $\partial N$ defined by $h(W(k))$ is  regular (=Whitney) stratification.
All  rounding-corners homeomorphism provide a unique smooth structure of smooth manifold with boundary on $(W,\partial W)$ \footnote {Precisely,  for two such rounding-corners homeomorphisms $h_i: (W,\partial W)\to (N,\partial N),$ $i=1,2$ there exists diffeomorphisms $\lambda: N_1 \to N_2,$  unique up to isotopy s.t. $\lambda \cdot h_2$ and $h_1$ are isotopic as homeomorphisms}. 

\section {The completion / compactification theorem} \label {S6}

Recall that given $x, y\in \mathcal X$ and $x\ne y$ one writes $x>y$ iff $\mathcal T(x,y)\ne \emptyset$ 
and if the vector field $X$ is a M-S vector field  one defines 
 $i(x,y):= i(x)-i(y).$
\vskip .1in

Suppose $X$ is an M-S vector field.

\begin{itemize}
           
          \item A point $p\in W^{\mp}_x$ can be specified by the restrictions to $(-\infty.0]$ resp. $[0,\infty)$ of the trajectory $\gamma_p (t)$ with $\lim_{t\to -\infty}=x$ resp.  $\lim_{t\to +\infty}=x ,$  restrictions denoted by $\gamma^{\mp}_p,$ and a point $p\in \mathcal M(x,y)$ can be specified by an instanton $\gamma\in \mathcal T(x,y)$ with a marked point $p\in\gamma.$      
See Figure 1 (a).        
                    
            \item  A {\it $k-$broken instanton} from $x \in \mathcal X$ to $y \in \mathcal X$  at breaking points $y_1 , y_2, \cdots y_k \in \mathcal X,$\  $ x  >y_1 >\cdots y_k >y$ is an element of the set $$\mathcal T(x,y_1) \times \mathcal T(y_1,y_2) \cdots \mathcal T(_{k-1},y_k)\times \mathcal T(y_k,y),$$ hence a juxtaposition of  instantons $(\gamma^1, \gamma^2, \cdots, \gamma ^{k+1}), \gamma^i\in \mathcal T(y_{i-1}, y_i).$ {\it Juxtaposition} here means
 $ \lim_{t\to \infty} \gamma^i(t)= \lim_{t\to -\infty} \gamma^{i+1}(t)= y_i,$ with $y_0= x$ and $y_{k+1}=y$  cf. Figure 1. (b).

           \item  A $k-$broken trajectory  from $x \in \mathcal X$ to $p\in M\setminus \mathcal X$  at breaking points $y_1 , y_2, \cdots, y_k\in \mathcal X,$ 
$ x >y_1 >\cdots > y_k ,$ is an element of the set $$\mathcal T(x,y_1) \times \mathcal T(y_1,y_2) \cdots \times \mathcal T(_{k-1},y_k)\times W^-_{y_k},$$ hence a juxtaposition  $(\gamma^1, \gamma^2, \cdots, \gamma ^{k}, \gamma^-_p), \gamma^i\in \mathcal T(y_{i-1}, y_i),$  $p\in W^-_{y_k},$ $y_0= x$, cf. Figure 1 (c).

 Similarly, 
 a $k-$broken trajectory  from  $p\in M \setminus \mathcal X$ to $y\in \mathcal X$   at breaking points $y_1 , y_2, \cdots y_k\in \mathcal X,$ $y_1>y_2\cdots y_{k} > y,$ is an element of the set $$\mathcal  W^+_{y_1}\times \mathcal T(y_1,y_2) \cdots \times \mathcal T(y_{k},y),$$ hence a juxtaposition  $(\gamma_p^+,\gamma^{1},  \cdots, \gamma ^{k+1}),$ $p\in W^+_{y_1},$ $\gamma^i\in \mathcal T(y_{i-1}, y_i),$  $y_{k+1}=y,$ cf. Figure 1 (d).
 
          \item A pointed $k-$broken instantons from $x\in \mathcal X$ to $y \in \mathcal X$  at breaking points $y_1 , y_2, \cdots y_k \in \mathcal X,$  $ x >y_1 >\cdots y_k >y,$ with marked point $p\in \mathcal M(y_r, y_{r+1}),$ is an element of the set 
$$
 \mathcal T(x, y_1)\cdots  \times \mathcal T(y_{r-1},y_r) \times \mathcal M(y_r, y_{r+1})\times \mathcal T(y_{r+1},y_{r+2})\cdots \times   \mathcal T(y_k,y).$$
 
It can be also represented by a pair 
of two elements the first  
$(\gamma^1, \gamma^2, \cdots \gamma ^{r}, \gamma^-_{p})$  
and $(\gamma^+_p,\gamma^{r+2} \cdots, \gamma ^{k+1}),$  $p\in \mathcal M(y_r, y_{r+1}),$ $x=y_0,$ $y= y_{k+1},$ cf. Figure1 (e). 
\vskip .1in

        \item A pointed   $k-$broken instantons from $x\in \mathcal X$ to $y \in \mathcal X$  at breaking points $y_1 , y_2, \cdots y_k,$  with marked point $p=y_r\in \mathcal X$ $r=0,1,\cdots y_{k+1},$ is an element of the set 
$$
= \mathcal T(y_0,y_1)\cdots  \times \mathcal T(y_{r-1},y_r) \times \mathcal M(y_r, y_r)\times \mathcal T(y_r, y_{r+1})\times \cdots \times \mathcal   \mathcal T(y_{k},y_{k+1})
,$$ $y_0=x, \ y_{k+1}=y.$  

Note that for any $y_r\in \mathcal X,$ \  $\mathcal M(y_r, y_r)= \{y_r\}.$ 
 \end{itemize}
\hskip .5in
\begin{figure}[h]
\center
\includegraphics [height=8.5cm]{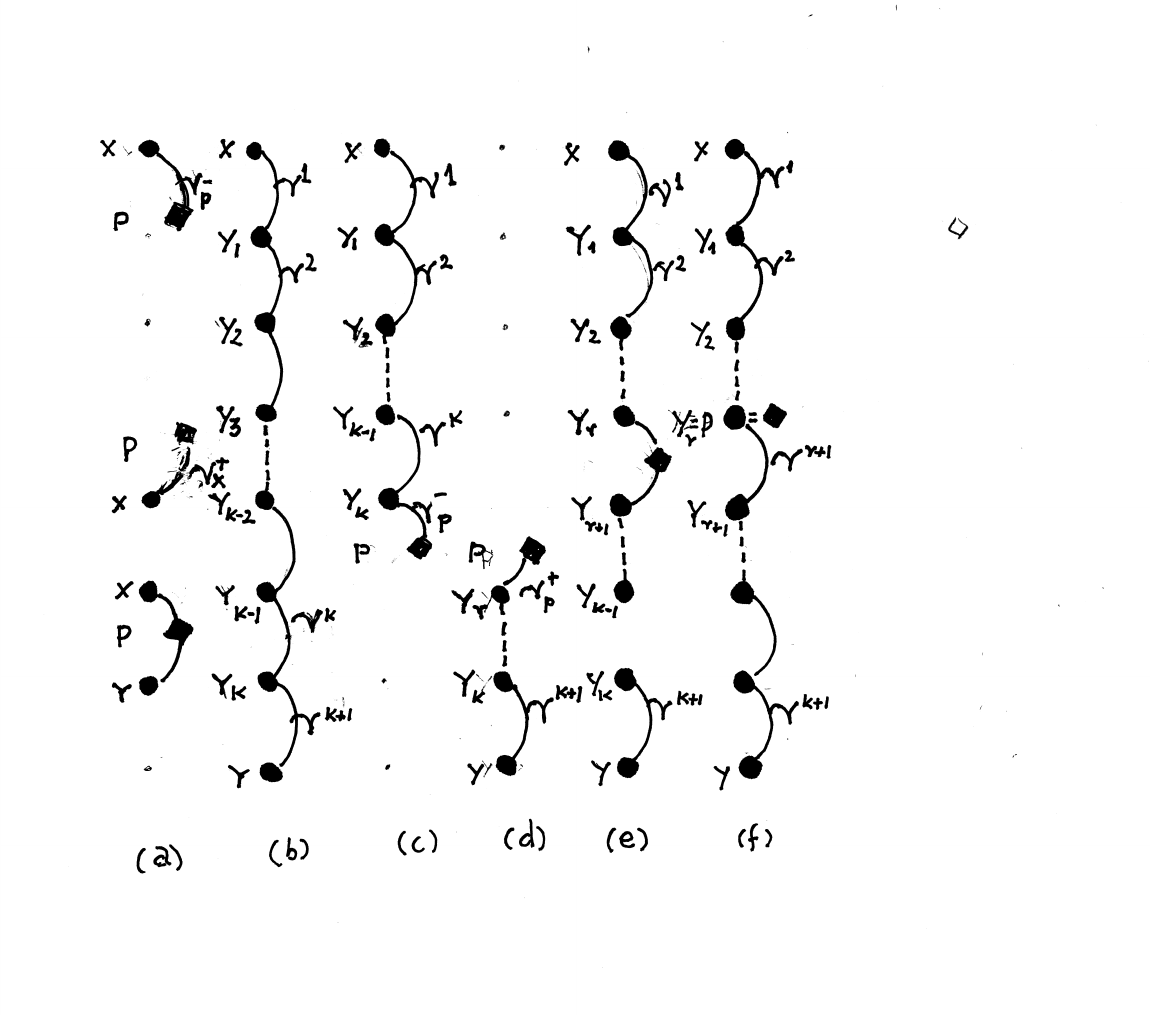}
\caption{a,c.b,d,e,f}
\end{figure}

\vskip .1in
Define :
\begin{itemize}

\item $\hat{ \mathcal T}(x,y)$ by 
\begin{enumerate}
\item $\hat {\mathcal T}(x,y)(0)= \mathcal T(x,y),$
\item $\hat {\mathcal T}(x,y)(k):= \bigsqcup_{\{ y_1,\cdots,y_k\mid x >y_1>\cdots, y_k >y\}} \mathcal T (x ,y_1)\times \mathcal T(y_1, y_2)\times \cdots \mathcal T( y_k,y)$
\item $\hat {\mathcal T}(x,y): = \sqcup _{\{k= 0,1, \cdots, i(x,y)-1\} }\hat {\mathcal T}(x,y)(k). $
\end{enumerate}

\item  $\hat W^-_x, \hat i^-_x: \hat W^-_x \to M$  by 
\begin{enumerate} 
\item $\hat W^-_x(0)= W^-_x,$  
\item  $\hat W^-_x(k):= \bigsqcup_{\{y_1 >y_2 \cdots  > y_k \mid x>y_1\}} \mathcal T(x,y_1) \times \mathcal T(y_1,y_2) \cdots \mathcal T(_{k-1},y_k)\times W^-_{y_k},$  
\item $\hat W^-_x: = \sqcup _{k= 0,1, \cdots, i(x)} \hat W^-_x (k), $ 
\item ${\hat i^-_x} |_{\hat W^-_x (k)}:= i^-_{y_k} \cdot p_{W^-_{y_k}}$  with $p_{W^-_{y_k}}$ the projection on $W^-_{y_k}.$  
\end{enumerate}

\item  $\hat W^+_y, \hat i^+_y: \hat W^+_y \to M$  by 
\begin{enumerate} 
\item $\hat W^+_y(0)= W^+_y,$  
\item  $\hat W^+_y(k):= \bigsqcup_{\{y_1 >y_2 \cdots  > y_k \mid  y_k>y\}} W^+_{y_1} \times \mathcal T(y_1,y_2) \times \mathcal T(y_1,y_2) \cdots \mathcal T(y_k,y),$  
\item $\hat W^+_y: = \sqcup _{\{k= 0,1, \cdots, n-i(y)\}} \hat W^+_y (k), $ 
\item $\hat i^+_y|_{\hat W^+_y(k)}:= i^+_{y_1} \cdot p_{W^+_{y_1} }$  with $p_{W^+_{y_1}}$ the projection on $W^-_{y_1}.$  
\end{enumerate}

\end{itemize} 

In order to describe $\hat{\mathcal M}(x,y),$  and $\hat i_{x,y} : \hat {\mathcal M}(x,y)\to M$  consider first 

\begin{enumerate} [label=(\alph*)]    
         \item $\mathcal M(x,x):= \{x\},$ 
         \item  $\mathcal M(y_0 >y_1 >\cdots >y_{k+1}):= \sqcup_{r=0,1,\cdots, k}  \mathcal T(y_0,y_1)\cdots  \times \mathcal T(y_{r-1},y_r) \times \mathcal M(y_r, y_{r+1})\times \mathcal T(y_{r+1},y_{r+2}) \times \cdots \mathcal T(y_k,y_{k+1}),$ 
         \item  $\mathcal M'(y_0 >y_1 >\cdots >y_{k+1}):= \sqcup_{r=0,1,\cdots k+1}  \mathcal T(y_0,y_1)\cdots  \times \mathcal T(y_{r-1},y_r) \times \mathcal M(y_r, y_r)\times\mathcal T(y_r, y_{r+1}) 
 \times \cdots 
 \mathcal T(y_{k},y_{k+1}),$  
\end{enumerate}
 \hskip .2in and define 

\begin{itemize}
\item $\hat{\mathcal M}(x,y),$  and $\hat i_{x,y} : \hat {\mathcal M}(x,y)\to M$ by 
\begin{enumerate}
\item $ \hat {\mathcal M} (x,y) (0):= \mathcal M$
 \item  $ \hat {\mathcal M} (x,y) (k):=  \begin{cases} \bigsqcup _{\{y_1 >\cdots >y_{k}
 \mid x>y_1, y_k >y\}}\mathcal M(x >y_1 >\cdots y_{k} > y) \  \sqcup \\
 \bigsqcup _{\{y_1 >\cdots >y_{k}\mid y_0=x>y_1, y_k >y_{k+1}=y\}}\mathcal M'(y_0 >y_1 >\cdots >y)\end{cases}$
 \item $ \hat {\mathcal M}(x,y):=\sqcup _{k=0,1\cdots i(x,y)} \hat{\mathcal M }(x,y)(k)$
 \item  $\hat i_{x,y}$ restricted to any component of $\hat {\mathcal M}(x,y),$  is the  composition of the projection on the $\mathcal M(\cdots)$  followed by the injective maps   $i_{\cdots}.$
\end{enumerate}
\end{itemize}

Since $X$ is an M-S vector field  all $\hat W^\mp_x (k), \hat {\mathcal M} (x,y) (k), \hat {\mathcal T}(x,y)(k)$ are smooth manifolds of dimension $\dim W^\pm_x -k, \dim \mathcal M(x,y)-k, \mathcal T(x,y)-k$ respectively. 
In view of the above definitions  one has:

\begin{obs}\label{O51}\

For any $x,y\in \mathcal X, x>y$  the set $\hat{\mathcal M}(x,y)$ identifies to the set of marked instantons  from $x$ to $y$ which is the subset  $(\hat i^-_x \times \hat i^+_y)^{-1} (\Delta(M))$
 of $\hat W^-_x\times  \hat W^+_y$ as described by the  commutative diagram
\begin{equation} \label {D1}
\xymatrix{ \hat W^-_x\times \hat W^+_y \ar[r]^{\hat i^-_x\times \hat i^+_y} & M\times M  \\
\hat{\mathcal M}(x,y)\ar[u]^{\subseteq} \ar[r]^{\hat i_{x,y}} &M\ar[u]^{\Delta}}\end{equation} 
with $\Delta$ the diagonal map. 
\end{obs}

\begin{theorem}  \label {T62} (cf. \cite{Bu3},\cite {BFK}, \cite {BH}) \              
\begin{enumerate}
\item Suppose $X$ is an M-S vector field. Then $\hat W^\mp_x, \hat {\mathcal M}(x,y), \hat{\mathcal T}(x,y)$  have a structure of orientable smooth manifolds with corners whose $k-$corners are 
$\hat W^\mp_x (k), \hat {\mathcal M}(x,y)(k), \hat {\mathcal T}(x,y)(k)$  and  $\hat i^\mp_x$ and $\hat i_{x,y}$ described above are smooth maps.  Moreover  all these manifolds  have stably trivial tangent bundle
and the maps ${\hat i}_{\cdots} $ restricted to any component of the $k-$corner are of constant rank and submersion over their image in $M.$
\item If  $f:M\to \mathbb R$ is a   Lyapunov function  for the vector field $X,$  which is either proper and bounded from below or the lift of a closed differential one form  on a closed manifold  (cf. section 2),  
then  each $\hat W^\mp_x, \hat {\mathcal M}_{x,y}, \hat {\mathcal T}(x,y)$  is compact.
\end{enumerate}
\end{theorem}
\vskip .in
\begin{proof}

1.  The proof  that $(\hat W^-_x, \hat W^-_x(k))$ and  $(\hat {\mathcal T}(x,y), \hat {\mathcal T}(x,y) (k))$are manifold with the corners and $\hat i^-_x$ and $\hat i_{x,y}$ are smooth maps 
with the properties as stated in Theorem \ref{T62} 
was done in details in \cite{Bu2}, \cite{BFK} and \cite{BH}. The same remains true  for  $(\hat W ^+_x, \hat W^+_x(k))$ and $i^+_x$ in view of the observation that $W^+_x,$ with respect to the vector field $X,$ is actually $W^-_x$ with respect to  the vector field $-X,$ which remains M-S  with $-f $ as Lyapunov function.
The statements  for $\hat {\mathcal M }(x,y)$ and $\hat i_{x,y}: \hat{\mathcal M}(x,y)\to M$  follow from Observation \ref{O51}, 
once one verifies the transversality  $ \hat i^-_x\times \hat i^+_y \pitchfork \Delta(M).$ Indeed, the transversality at  $p= \Delta(x) \in \Delta(M),$  $x\in \mathcal X,$  follows from the transversality of $W^-_x$ and $W^+_x$ and at $p\in \Delta (\mathcal M(y_{r-1}, y_r))$   from the transversality of $W^-_{y_{r-1}}$ and $W^+_{y_r}.$   
 
\vskip .1in 
 In order to establish that   $\mathcal  M(x,y)$ and then of $\mathcal T (x,y)$  have stably trivial  tangent bundle one proceeds as follows. 
One denotes by : 
 \begin{enumerate} 
  \item $\tau$ the tangent bundle of $\mathcal M(x,y),$
  \item $\nu^-$ the normal bundle  $\mathcal M(x,y)$ in $W^-_x,$
 \item  $\nu^+$ the normal bundle  $\mathcal M(x,y)$ in $W^+_y.$
\end{enumerate}
Since  the tangent bundles of  $W^-_x$ and $W^+_y$  restricted to $\mathcal M(x,y)$ are trivial of rank $i(x)$and $n-i(y)$ one has: 
\begin {enumerate} [label=(\roman*)]
\item $\tau\oplus \nu^-_x,$ 
the tangent bundle of $W^-_x$ restricted to $\mathcal M(x,y),$ is isomorphic to $\epsilon^{i(x)},$ 
\item $\tau\oplus \nu^+_y\sim \epsilon^{n-i(y)},$  the tangent bundle of $W^+_y$ restricted to $\mathcal M(x,y),$ is isomorphic to $\epsilon^{n-i(y)},$ 
\end{enumerate} 
with $\epsilon ^k$ denoting the trivial vector bundle of rank $k.$

In view  of transversality $W^-_x\pitchfork W^+_y$ one has
$ \epsilon^{i(x)}\oplus \epsilon ^{n-i(y)}\simeq T(M) |_{\mathcal M(x,y)}\oplus \tau \sim \epsilon ^n \oplus \tau$  
\footnote {since $\mathcal M(x,y)$ is a subset of contractible space}.

Items (i) and (ii) above  implies 
$\epsilon ^{n+i(x)-i(y)} \sim \epsilon ^n\oplus \tau,$ which implies that $\tau$ is stably trivial, hence $\mathcal M(x,y)$ is a stably parallelizable manifold, hence $\hat{\mathcal M}(x,y)$ is a stably parallelizable manifold. The same holds for $\mathcal T(x,y)$ and $\hat {\mathcal T}(x,y).$ 
\vskip .2in 

2. The compacity statements follow from the compacity of $\hat W^\pm_x$ , $\hat {\mathcal M}(x,y)$ and of $\hat {\mathcal T} (x,y)$ established in   \cite  {Bu3},  \cite {BFK}, or  \cite{BH}.
 \end{proof} 
   
\section {Orientability and orientations  associated with a M-S vector field} \label{S7}

For a rank $k$ real vector bundle $\xi: E\to M$ over a space $M$ denote by $\Lambda (\xi)$ the line bundle $\Lambda (\xi):= \Lambda^k(E)\to M.$  Recall that one  calls the bundle $\xi$ orientable iff $\Lambda(\xi)$ has nonzero sections (equivalently is trivial), and in this case, an equivalence class of nonzero sections \footnote{One says that the  sections $s_1$ and $s_2$ are equivalent iff $s_1= f \cdot s_2$ for some  $f$ a positive  real-valued function}
is called {\it orientation} and denoted by $o.$ 
If the vector bundle $\xi$ is orientable  then  its dual $\xi^\ast$ is orientable, and an orientation $o$ for $\xi$ determines an orientation $o^{-1}$ of $\xi^\ast.$ 

If the space $M$ is connected, and the vector bundle $E\to M$ is orientable an orientation $o(p)$ for the vector space $E_p, p\in M$  determines and is determined by an orientation $o$ for the vector bundle $E\to M.$ Clearly  only two orientations are possible; if one, $o,$ is  represented by the section $s$ in $\Lambda(\xi)$ then the other one (the opposite, $-o,$) is represented by the section $-s.$ 
Note that $\Lambda(\xi_1 \oplus \xi_2)= \Lambda (\xi_1)\otimes \Lambda (\xi_2).$ Then, in consistency with this formula the orientations $o_1$ for $\xi_1$ and $o_2$ for $\xi_2$ determine the orientation $o=
 o_1\cdot o_2$ for $\xi_1\oplus \xi_2.$

If $\xymatrix {0\ar[r] &E_1\ar[r]^i &E\ar[r]^\pi&E_2\ar[r]&0}$ is a short exact sequence of vector bundles then a splitting  $s:E_2\to E$  or a projection $\pi': E\to E_1$  provides an isomorphism of vector bundle $E_1\oplus E_2\to E$  or $E\to E_1\oplus E_2$. Such isomorphisms  although non canonical, since they  depend on $s$ or on $\pi',$  ultimately  provides a canonical identification  
of the orientation $o_1\otimes o_2,$  with $o_1$ and $o_2$ orientations for $E_1$ and $E_2,$ to an orientation $o$ for $E,$ 
equivalently written 
$$ o_1\cdot o_2= o, \ \ o_1= o \cdot o_2^{-1},\ \ o_2= o_1^{-1} \cdot o.$$ 
In particular the orientability / orientations of two of the vector bundles $E_1, E_2, E$ determine the orientability / orientation in the third.
Note that the orientation $o\otimes o^{-1}= o\cdot o{-1},$ induced by the canonical isomorphism $\Lambda(\xi)\otimes \Lambda(\xi^\ast) = \epsilon,$ is the canonical orientation for the trivial line bundle $(\epsilon: \mathbb R\times M\to M$) given by the section provided  by  the constant map equal to $1$ on $M.$

With the notations above  the isomorphism $E_1\oplus E_2\to E_2\oplus E_1,$  $(x,y)\to (y,x),$  implies $o_1\cdot o_2= (-1)^{k_1 k_2} o_2 \cdot o_1$ where $k_i= \rank E_i.$
\begin{proposition} \label {P71}\ 

Two orientations $o_x^-$ and   $o_y^-$ for the unstable manifolds $W^-_x$ and $W^-_y,$  $x,y\in \mathcal X$ with  $x>y,$ 
induce an  orientation $o_{x,y}$ for the manifold $\mathcal M(x,y).$
\end{proposition}
\begin{proof}

For any $x,y\in \mathcal X,$ denote by: 
\begin{enumerate} [label=(\alph*)]
\item $T,$ the tangent bundle of $M,$ 
\item $\tau_x^-$ 
the tangent bundle of $W^-_x,$ 
\item $\tau^+_y,$ the tangent bundle of $W^+_y$ 

\noindent and continue to use the same notations for the restrictions of these bundles to $\mathcal  M(x,y),$  
\item $\tau,$ the tangent bundle of $\mathcal M(x,y),$
\item $\tau_y^-,$ 
the tangent bundle of $W^-_y,$ 
\end{enumerate}

In view of the transversality $W^-_x\pitchfork W^+_y$ one has:
the short exact sequence of vector bundles over $\mathcal M(x,y)$ 
\begin{equation} \label {E4}
0 \to \tau \to \tau^-_x\oplus \tau^+_y \to T\to 0,\end{equation} and the isomorphism of vector spaces 
 $\tau^+ _y (y)\oplus \tau^-_y(y)= T(y).$

Consider the orientations $o_x$ and $o_y$ for $\tau^-_x$ resp. $\tau^-_y$ and choose the orientation $o$ for $T |_{\mathcal M (x,y)}.$ 

Since $W^+_y$ is orientable and connected the orientations $o$ and $o_y$  provide  an orientation   $o^+(o, o_y)$ for $W^+_y$ s.t.  $- o^+(o,o_y)= o^+(-o,o_y)$.
In view of (\ref{E4}) one obtains  
$ o_\tau\cdot o = o_x\cdot o^+_y(o, o_y)$ which in view of the equality above makes $o_\tau$ independent on $o.$ 
\end{proof} 

\vskip .2in

{\bf A few consequences} 
\vskip .1in
Suppose $X$ is a M-S vector field and  $\mathcal O=\{o_x\}$ a collection of orientations. The collection $\mathcal O$   
 induces  by Proposition (\ref{P71} the collection of orientation  $\{o_{x,y}\}$  on $\mathcal M (x,y)$  and  
in view of the free action $\varphi: \mathbb R \times \mathcal M(x,y)\to \mathcal M(x,y),$ whose quotient space is the manifold $\mathcal T(x,y)$ provides
\begin{enumerate}  
\item an orientation  on $\mathcal T(x,y)$ and implicitly  on each component of $\mathcal T(x,y),$ in particular
\item a sign-orientation $\epsilon(\gamma)\in \{\pm 1\}$ for each component $\gamma$ of $\mathcal T(x,y)$ when $i(x)-i(y)=1,$  precisely $\epsilon (\gamma)=+1$ if the induced orientation on $\gamma$ is from $x$ to $y$ and $\epsilon (\gamma)= -1$ otherwise.   
\end{enumerate}
As a straightforward consequence of Theorem (\ref {T62}) , Proposition  (\ref{P71}) and of Stokes's theorem one has:  
\begin {obs} \label{O72}\

If the manifold $\mathcal M(x',y')$ is part of $\hat {\mathcal M}(x,y) (1)$  then 
either one of the two situations hold true
\begin{enumerate}
\item $x'=x$ and $x  >\ \rm{or}\ =  y'>y,$ with $i(y')= i(y)+1,$  in which case the orientation $o_{x,y}$ induces on $\mathcal M(x,y')\times \gamma$ the orientation  $\epsilon (\gamma) o_{x,y'}$  or
\item $y'=y$ and $x>x' \rm{or} =y,$ with $i(x')= i(x) -1,$  in which case the orientation $o_{x,y}$ induces  on $\gamma \times \mathcal M(x', y)$ the orientation  $(-1)^{i(x,y')+ \epsilon (\gamma)}  o_{x',y} $ \ .
\end{enumerate}
\end{obs}

\begin{obs}\label {O73}\

Under the hypotheses of Theorem 6.2 item 2, in the presence of the collection of orientations $\mathcal O,$  the following holds true:
\begin {enumerate}
\item For any $x,y\in \mathcal X$ with $i(x,y)= i(x) - i(y)=p$  and   $\omega\in \Omega^p(M)$ then the integral $\int_{\mathcal M(x,y)} \omega$ is convergent  and
equal to $\int_{\hat {\mathcal M}(x,y)} \hat i_{x,y} ^\ast \omega.$ 
\item  If  in addition $\omega |_{\hat W^-_x}= d\alpha,$  $\alpha\in \Omega^{p-1}(\hat W^-_x),$ then 
$$\int_{\mathcal M(x,y)} \omega= \sum_ {\{y'\mid i(y')= i(y)+1\}} \int_{\mathcal M(x,y')} \alpha + (-1)^{i(x,y)-1}\sum_ {\{x'\mid i(x')= i(x)-1\}}  \int_{\mathcal M(x,y')} \alpha.$$
\end{enumerate}
\end{obs} 
\vskip .2in 

\section {Dynamics, cohomology and integration of forms }\label {S8}

Let  $X$ be a M-S vector field on the smooth manifold $M$  and suppose that the following two properties are satisfied.

{\bf C1}: for any $x,y\in \mathcal X$  the spaces $\hat W^\pm_x$ and $\hat {\mathcal M}(x,y)$ are compact,  

{\bf C2}: for any $x\in \mathcal X_r$  the set $\{y\in \mathcal X\mid x>y\}$ is finite \footnote {actually the compacity of all $W^\pm_x$ for all $x\in \mathcal X$ is equivalent to the compacity of all $\mathcal M(x,y)$ for all $x>y, x,y\in \mathcal X$ hence C2 implies  C1 }, 

Note that both  C1 and C2  are satisfied in case that $f$ is a proper Lyapunov function bounded from below 
or is the  lift  
of a Lyapunov (for $X$) closed one form 
on a  closed manifold as described in section 2.  
\vskip .1in 
Consider :

\begin{enumerate} [label=(\alph*)]
\item  for $x\in \mathcal X_r, y\in \mathcal X_{r-1}$  the map  $$\boxed{I_r : \mathcal X_r\times \mathcal X_{r-1}\to \mathbb Z} $$ well defined  in view of C2  by 
$$\mathbb I_r (x,y):=    
 \sum _{\gamma\in \mathcal T(x,y)} \epsilon (\gamma),$$   
\item for any $x\in \mathcal X$  and differential form $\omega \in \Omega^{i(x)}(M)$ 
 the map 
$$\boxed{Int_r : \Omega^r (M) \times \mathcal X_r\to \mathbb R}$$ well defined in view of C1 by $Int(\omega, x):= \int_{\hat {\mathcal W}^-_x} i^\ast_x \omega=\int_{\mathcal W^-_x}  \omega ,$  

\item For any $x\in \mathcal X_{r+p}, y\in \mathcal X_p$  and differential form $\omega \in \Omega^r(M)$   the map 
$$\boxed{Ent_{r,p} : \Omega^r (M)\times \mathcal X_p \times \mathcal X_{p+r}\to \mathbb R}$$ well defined in view of C1 by $Ent_{r,p}(\omega, y,x):= \int_{\hat {\mathcal M}(x,y)} i^\ast _{x,y}\omega= \int_{\mathcal M(x,y)} \omega,$
\item $C^r(X):= Maps (\mathcal X_r, \mathbb R).$ 
\end{enumerate}
\vskip .1 in

The maps $\mathbb I_r : \mathcal X_r \times \mathcal X_{r-1}\to \mathbb Z$, $Int_r : \Omega^r(M)\times  \mathcal X_r\to \mathbb R$ and 
$Ent_{r, p} : \Omega^r(M)\times  \mathcal X_p\times \mathcal X_{p+r}\to\ \mathbb R$  induce  the linear maps 

\begin{enumerate} 
\item $\delta^r: C^r(X)\to C^{r+1}(X)$ defined by 
$$\delta^r(f) (x):= \sum_{y\in \mathcal X_r} I_{r+1} (x,y) f(y)$$   for $x\in \mathcal X_{r+1},$
\item $Int_r: \Omega^r(M)\to C^r(X)$ defined by $$Int_r (\omega) (x)=\int_{W^+_x} \omega  $$  for $x\in \mathcal X_r,$
\item $E_{r,p} :\Omega^r(M) \otimes C^p(X)\to C^{p+r}(X)$ defined  by $$E_{r,p} (\omega \otimes f) (x):= \sum_{\{y\in \mathcal X_p \mid x>y\}}  f(y)Ent_{r.p} (\omega, y,x)
= \sum_{\{y\in \mathcal X_p\mid x>y\}}  f(y)\int_{\mathcal M(x,y)}  \omega
$$  for any $x\in \mathcal X_{p+r}.$
\end{enumerate}

\begin{proposition} \label {P81}\ 

Suppose $X$ is M--S vector field which satisfies C1 and C2 and  
 $\mathcal O= \{o_x\}$ is a collection of orientations. 
 The following holds true:
  \begin{enumerate}
\item $\delta^{r+1}\cdot \delta_r=0,$ equivalently,  for any $x\in \mathcal X_{r+1}, z\in \mathcal X_{r-1}$ one has $\sum _{y\in \mathcal X_r} \mathbb I_{r+1}(x,y)\cdot \mathbb I_r(y,z)=0,$
\item  $\delta^r \cdot Int _r (\omega)= Int_{r+1} (d \omega),$  
\item  $\delta ^{r+p}(E_{r,p}(\omega \otimes f)) (x)= Ent_{r+1,p}(d\omega, f) (x)  + (-1)^r E_{r, p+1}(\omega, \delta^p f) (x),$  

for $ x\in \mathcal X_{r+1},\  \omega\in \Omega^r(M),\ f\in C^p(X).$
\end{enumerate}
\end{proposition}

\begin{proof}
 
Item 1.: If $i(x)- i(z)=2$ then $\hat {\mathcal T}(x,z)$ is a compact one dimensional manifold  with boundary, hence it is (diffeomorphic to) a finite union of compact oriented intervals, Then the  boundary 
of $\hat {\mathcal T}(x,y)$ consists of the same number of  points with the sign-orientations $"+ 1``$ and of points with the sign-orientation $"-1``,$  whose total sign-cardinality is $0.$ This cardinality, in view of the definition of $\delta_r,$  is actually $\sum_ {\{y\in \mathcal X \mid x >y >z\}} I_{i(x)} (x,y ) I_{i(x)-1} (y,z),$ hence $=0.$  

In order to check items 2. and 3. recall from Theorem (\ref {T62}) that the only components of $\hat W^-_x(1)$ whose image by $\hat i_x$ have dimension $i(x)-1$  are 
$\mathcal T(x,y)\times \mathcal W^-_y$ for $x >y$ and $i(y)= i(x)-1$ and 

the only components of $\hat {\mathcal M}(x,y)(1)$ whose image by $\hat i_{x,y}$ have dimension $i(x,y)-1$  are 
$\mathcal T(x,x')\times \mathcal M(x',y)$ for $x>x'>y$ with $i(x')= i(x)-1$ and $\mathcal M(x,y')\times \mathcal T(y',y)$ for $x>y'>y$ with $i(y')= i(y)+1.$  In view of Stokes's theorem and  
of Observations \ref{O72} and \ref{O73} one has the following:   

\begin{equation}\label {E6}
\int _{W^-_x} d \omega = \sum _{\{y \mid i(y)= i(x)-1\}} I(x,y) \int _{W^-_y}  \omega \\
\end{equation} and 
\begin{equation}\label {E7}
\int _{\mathcal M(x,y)} d \omega = \sum _{\{x' \mid i(x')= i(x)-1)\}}I_{i(x)}(x, x') \int _{\mathcal M(x',y} \omega
    + \sum _{\{y' \mid i(y')= i(y)+1\}}    (-1)^{i(x,y)-1}I_{i(y)}(y,y') \int _{\mathcal M (x,y')}  \omega          
\end{equation}

Item 2. follows immediately from (\ref{E6})  and item 3. form  (\ref{E7}).
The verification of item 3. is eased by  the following  intermediate calculations:

\begin{enumerate}[label=(\alph*)]
\item 
\begin{equation} \label {E8}
\delta (E_{r,p} (\omega, f)) (x)= \sum _{x'\mathcal X_{r+p}} \sum _{y\in \mathcal X_p} f(y) \mathbb I(x,x') \int _{\mathcal M(x',y)}\omega,
\end{equation}
\item \begin{equation}
E_{r+1,p} (d\omega, f)(x)= \begin{cases}
&\sum_{y\in \mathcal X_p}\sum_{x' \in \mathcal X_{r+p}}  f(y) I(x,x') \int _{\mathcal M(x',y)} \omega   + \\
(-1)^{r+1}&\sum_{y\in \mathcal X_p} \sum _{y'\in \mathcal X_{p+1}}  f(y)I(y',y)  \int_{ \mathcal M(x,y')} \omega
\end{cases},
\end{equation}

\item \begin{equation} \label{E9} 
\begin{aligned}(-1)^r E_{r, p+1}(\omega, \delta  f) (x)=&(-1)^r \sum _{y' \in \mathcal X_{p+1}} \delta f(y')\int _{\mathcal M(x,y')} \omega =\\
&= (-1)^r \sum _{y'\in X_{p+1}} \sum _{y\in \mathcal X_p} I(y',y) f(y) \int _{\mathcal M(x,y')} \omega. 
\end{aligned}
\end{equation}

\end{enumerate}

Clearly 
(7), (8) and (9) imply (\ref{E7}).

\end{proof}

\vskip .2in

For three cochain complexes $\mathcal C^\ast_i=(C^\ast_i, d^\ast_i), i=1,2,3, $  a morphism of cochain complexes $\Lambda: C^\ast_1\otimes C^\ast _2 \rightarrow C^\ast_3$ 
is provided by the linear maps $\Lambda_{p,q}: C^p_1\otimes C^q_2 \to
C^{p+q}_3$ s.t.  $$d_3 (\Lambda_{p,q} (f\otimes g))= \Lambda_{p+1,q} (d_1f\otimes g) + (-1)^p \Lambda_{p,q+1} (f\otimes d_2 f)$$ which induces, by passing to cohomology, the linear maps $\Lambda_{r,p} : H^r(\mathcal C_1) \otimes H^p(\mathcal C_2) \rightarrow  H^{r+p}(\mathcal C_3).$

We apply this to the cochain complexes $\mathcal  C^\ast_1=(\Omega^\ast (M), d^\ast),$ the de-Rham complex of $M,$  $\mathcal C^\ast_2$ and  $\mathcal C^\ast_3$ the cochain complex $(C^\ast(X), \delta ^\ast)$ and derive the linear maps
$E_{p,q}: H^p_{DR} (M)\otimes H^q_{X} (M)\to H^{p+q}_{X} (M),$ where  $H^\ast_{DR} (M)$ denotes  the  de-Rham cohomology, $H^\ast_{X} (M)$  the cohomology of the cochain complex$(C^\ast(X), \delta^\ast).$  Continue to denote by $Int_q: H^q_{DR}(M)\to H^q_X (M)$ the linear map induced by $Int_q.$ 

\begin{theorem} (elementary Morse-theory) \label {T82} \

If $X$ is a M-S vector field on a closed manifold which has Lyapunov function and $\mathcal O$ a collection of orientations 
\footnote  {considerably weaker hypotheses like proper Lyapunov function bounded from below suffice} 
then the linear maps  $Int_q:\Omega ^q(M)\to C^q(X)$ induce
the linear isomorphisms $Int_q: H_{DR}^q(M) \to H^q_X(M).$ 
\end{theorem}
For details  cf. \cite {BFK}.

\begin{theorem}\label{T83}
 Under the same hypotheses as in Theorem \ref{T82}, the possibly non-commutative diagram
$$\xymatrix{ \Omega^r(M) \otimes C^q(X) \ar[r] ^-{Ent_{r,q}}&C^{r+q}(X)\\
 \Omega^r(M) \otimes \Omega^q(M)\ar[r]^-{\wedge_{r,q}}\ar[u]^{id\otimes Int_q} &\Omega^{r+q}(M)\ar[u]^{Int_{r+q}} }$$
 induces, by passing to cohomology,
the commutative diagram
$$\xymatrix{ H_{DR}^r(M) \otimes 
H_{X}^q(M) \ar[r] ^-{E_{r,q}}&H_{M}^{r+q}(X)\\
H_{DR}^r(M) \otimes H_{DR}^q(M)\ar[r]^ -{\wedge_{r,q}}\ar[u]^{id\otimes Int_q} &H_{DR}^{r+q}(M)\ar[u]^{Int_{r+q}}}
 $$
with $\wedge$ the product induced from the wedge product of forms,  equivalently the analytic version of "cup-product``. 
\end{theorem}
\begin{proof}
In view of Proposition (\ref{P81}) all arrows in both  diagrams are well defined  with $Int_r: H_{DR}^r(M) \to H^r_X(M)$ isomorphism. In order to conclude that the second diagram is commutative 
it suffices to check that for any two cohomology classes $\xi_1\in H_{DR}^r(M),$ and $\xi_2\in H_{DR}^p (M)$ one can choose representatives, closed forms  $\omega_1\in \Omega^r(M)$ and $\omega_2\in \Omega^p(M),$ such that  
\begin{equation}\label {DD}
\int _{W^-_x} \omega_1 \wedge \omega_2 = \sum _{y\in \mathcal C_r\mid  x>y} \int _{W^-_y}\omega_1 \cdot  \int _{\mathcal M(x,y)}\omega_2
\end{equation}
holds true. This is indeed the case with any such representative $\omega_1$ vanishing in some  neighborhood of $\cup_{y\in \mathcal X<\leq r} W^-_y $
and $\omega_2$ vanishing in some  neighborhood of $\{x\}$  
choice always possible in view of the fact that $\cup_{\{y\in \mathcal X \mid i(y) <r\}} W^-_y $
is a compact ANR  of dimension $\leq  r-1$  
is a compact ANR  of dimension $\leq  p-1$.The verification  is sketched in the Appendix and   with  more details in an updated version of \cite {Bu3} soon to be posted on arXiv.

\end{proof}

\section {Proof of Theorem (\ref {T41})} \label {S9}

Part 1.  \

(a):
Since $f$ is strictly decreasing on each nontrivial trajectory there are no closed trajectories,
Since the set $\mathcal X$ is a discrete set of points in a closed manifold $M$ it is finite.
Since $\mathcal X$ is finite and each manifold  with corners $\hat {\mathcal T} (x,y)$ is compact then
$\pi_0(\mathcal T)= \pi_0(\hat {\mathcal T})$ is a finite set.

(b): holds true since  the cochain complex $(C^\ast(X), \delta_\ast)$ calculates the cohomology of $M$  by Theorem (\ref{T82}).

(c):observe that nontrivial cup product with $H^r(M)$ implies that at least for a pair of rest points $x>y$ with $i(x,y)=r$ $\int_{\mathcal M(x,y)} \omega \ne 0,$ hence $\mathcal T(x,y)\ne \emptyset$ 
hence $\mathcal T(r-1+p,p)\ne \emptyset.$

Part 2. \

One uses essentially the same arguments applied to the lift $\tilde f: \tilde M \to \mathbb R$  of the closed one form $\omega \in \Omega^1(M), d\omega=0$  on the closed manifold $M$ to the principal covering $\tilde M\to M$ discussed in section 2. The function $\tilde f$ is Lyapunov  for the vector field $\tilde X,$ the lift of $X$ on $\tilde M.$ 
 Details will be provided elsewhere.

\vskip .1in

\section {\bf Appendix}
\vskip .1in
Suppose  $X$ is a M-S vector field which satisfies C1, C2 in section 8 and $\mathcal O$ is a  collection of orientations. Then  the following holds true.

\begin {proposition}\label {P}
If $\omega_1$ and $\omega_2$ are two closed forms of rank $r$ and $p,$  the first  vanishing on a neighborhood of $\sqcup_{\{y\in \mathcal X\mid i(y) <r\}} W^-_y$  the second on a neighborhood of  $\sqcup_{\{x,y\in \mathcal X\mid x> y, i(x,y) <p\}} \mathcal M(x,y),$  then equality (\ref{DD}) holds true.
 \end{proposition} 

 First one specifies some notations.

For a space $X$ let 

\begin{itemize}
\item  $C(X) = X\times [0,1] / X\times \{1\}$ (with the space $X\times \{1\}$ collapsed to a point $v$),  
\item  $p_X:C(X)\to [0,1],$ the map defined by  $p(x,t)=t,$  hence $v=p^{-1}(1),$
\item   $C^1(X):= p^{-1} ([0,1))= X\times [0,1) = CX\setminus v$,  
\item $C^0(X):=CX\setminus \{(X\times 0) \sqcup v\}= X\times (0,1).$
\end{itemize}%
Clearly $C(X)\supset C^1(X)\supset C^0 (X).$ 

\noindent $CX$ is referred to as the {\it cone} of $X$, $v$ as the vertex of the cone and the space $X$ as the base  of the cone.

If $X$ is a smooth manifold then $C^0(X)$ is a smooth manifold, and $C^1(X)$ is a smooth manifold with boundary, whose  boundary is the base of the cone,  and map $p_X: C^1(X)\to [0,1)$ is smooth with any  $\epsilon\in [0,1)$  a regular value. 

 The unit disc $D^n$ in the Euclidean space $\mathbb R^n,$   is the cone of the unit sphere $S^{n-1},$ $D^n=C(S^{n-1},$ with the vertex $v= 0\in D^n$ and $p_S (u)= 1-||u||, u\in D^n.$ 
\vskip .1in

For the manifold $M$ equipped with the vector field $X$ as above 
 and any $x\in \mathcal X$ the validity of the Poincar\'e conjecture implies that there exists {\it rounding-corner} homeomorphisms (cf. section 5), $h_x: (\hat W^-_x,\partial \hat W^-_x) \to (D^{i(x)}, S^{i(x)-1}).$ Such homeomorphism  induces a {\it conic structure} for $\hat W^-_x,$ precisely a  homeomorphism $$\theta_x: C(\partial \hat W^-_x)\to \hat W^-_x$$ such that the restriction to any $\hat W^-(k)$ of the homeomorphism 
  $h_x\cdot \theta^{-1}_x$  
is a diffeomorphism onto its image.  One can choose $h_x$ s,t. $h_x(C^0 (y\times \mathcal T(x,y))= \mathcal M(x,y)$  

Recall from Theorem (\ref{T62}) that 

$\partial \hat W^-_x= \sqcup_{k\geq 1} \hat W^-_x(k)$,  

$\hat W^-_x(1)= \sqcup_{y\mid x>y}
 W^-_y\times \mathcal T(x,y), $  
 
 $\mathcal M(x,y)= \mathcal T(x,y)\times (0,1).$    

Define define, see Figure 2 below,
 
\begin{itemize}
\item $\boxed{ \hat W_{x,y}:= \theta_x (C(W^-_y\times \mathcal T(x,y))} \subset \hat W^-_x,$ 
\item $\boxed{\hat W^1_{x,y}:= \theta_x(C^1(W^-_y\times \mathcal T(x,y)))}= \theta_x(W^-_y\times \mathcal T(x,y)\times [0,1)) \subset $ open subset of  $\hat W^-_x\setminus x,$ which is 
a smooth manifold with boundary, whose boundary is diffeomorphic to  $W^-_y\times \mathcal T(x,y)),$ 
\item $\boxed{W_{x,y}:=\theta_x (\overset{\circ} {C}(W^-_y\times \mathcal T(x,y)))}= \theta_x(W^-_y\times \mathcal T(x,y)\times (0,1))$ diffeomorphic  to  $W^-_y\times \mathcal M(x,y)\subset W^-_x)$ which is an open subset of $W^-_x$
see Figure 2. below.
\end{itemize} 

\hskip .2in
\begin{figure}
\center
\includegraphics [ width=10 cm]{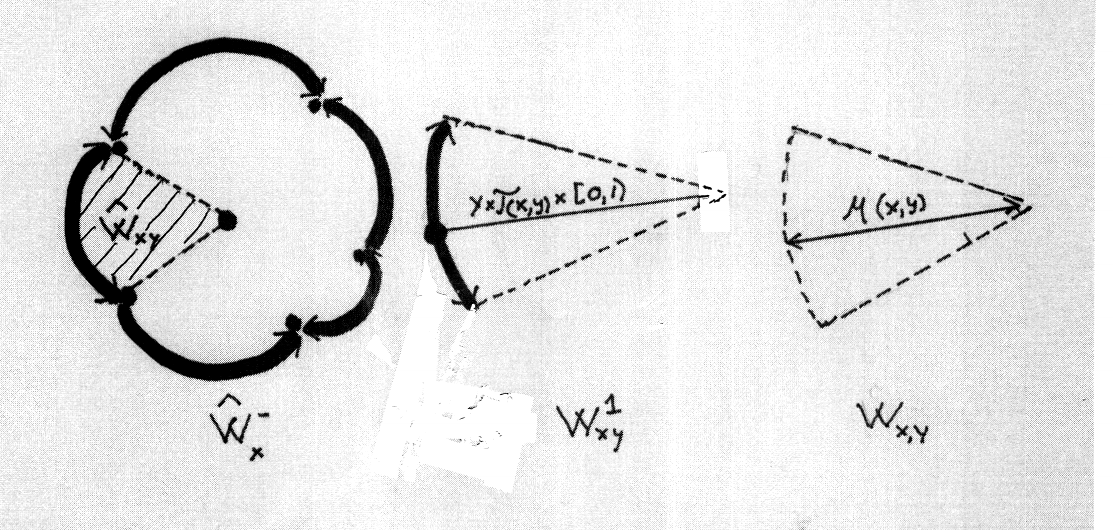}
\caption{}
\end{figure}

Note that :
\begin{enumerate}[label=(\alph*)]
\item Since the sets $\hat W^-_{x,y}$ are disjoint, a collection of differential closed forms  with compact support 
$\omega_ y  \in \Omega^k (\hat W^1_{x,y})$ define a differential closed form with compact support on $\hat W^-_x,$  $\omega\in \Omega^k(\hat W^-_x).$  
\item For a form $\omega \in \Omega^k(M)$ and $x\in \mathcal X$ with $i(x)=k,$  let $\underline \omega : i_x^\ast (\omega)\in  \Omega^k(\hat W^-_x)$ and observe that $\int _{W^-_x} \omega $ is convergent and equal to 
$\int _{\hat W^-_x} \underline \omega .$  
\item If $\underline \omega\in  \Omega^k(\hat W^-_x)$  is a closed form which when restricted  to $\partial \hat W^-_x$ vanishes, then   $\int _{\hat W^-_x} \underline \omega =0.$
This because  $\hat W^-_x$ is contractible hence $\underline \omega$  is exact and  Stokes Theorem applies;
in particular if $\underline \omega_i\in  \Omega^k(\hat W^-_x), i=1,2$ are two such closed forms whose restrictions to $\partial \hat W^-_x$ agree then $\int _{\hat W^-_x} \underline \omega_1 =
\int _{\hat W^-_x} \underline \omega_2.$
\end{enumerate}
 \vskip.1in
{\it Proof of Proposition (\ref{P}):}

Define 
$$\underline \omega_1:={\hat i_x}^\ast (\omega_1)\in \Omega^r(\hat W^-_x) , \ \ \  \  \underline \omega_2:={\hat i_x}^\ast (\omega_2)\in \Omega^p(\hat W^-_x)$$ 
and denote by 
$$\underline \omega_{1,y} := \underline \omega_1|_{\hat W^-_{x,y}}
\ \ \ \underline \omega_{2,y}:= \underline \omega_2|_{\hat W^-_{x,y}},$$   
their restrictions to $\hat W^1_{x,y}.$
\vskip .1in

Let $\pi_1$ and $\pi_2$ be the projection of $ W^-_y\times (\mathcal T(x,y)\times [0,1))$ on  the first and second factor and recall from section 3  the canonical identification of $\mathcal M(x,y)$ to the product $\mathcal T(x,y)\times (0.1).$ 
Define $$\underline \omega'_{1,y} := \pi_1^\ast (\underline \omega_{1,y}|_{\theta_x(W^-_y\times \mathcal T(x,y)\times 0)}) =\pi_1^\ast (\omega_1|_{W^-_y})$$ 
and 
$$\underline \omega'_{2,y}:= \pi_2^\ast (\underline  \omega_2 |_{\theta_x(y\times \mathcal T(x,y)\times [0,1)})$$  

These definitions simply mean  that for any $u\in W^-_y, \ \gamma \in \mathcal T(x,y), \  t\in [0,1)$ one has 
$$\underline \omega'_{1,y}(\theta_x(u,\gamma,t))=    \omega_1(u)$$  and 

$$\underline \omega'_{2,y} (\theta_x(y,\gamma,t))= \begin{cases}\omega_2(p), \rm{if}\ t\ne0\\ \omega_2(y), \rm{if} \  t=0\end{cases}, $$ with $p= (\gamma,t)$ 
 the point on the instanton $\gamma$ corresponding to $t$  by the identification of $ \mathcal M(x,y)$ with $\mathcal T(x,y)\times (0,1).$  when $t\ne 0$

One observes first that $\underline \omega'_{1,y} \wedge \underline \omega'_{2,y}$ is an $(r+p )-$ form with compact support in $\hat W^1_{x,y},$  hence the collection  for all $y $ with $x>y$ defines a closed form $\underline \omega'= (\underline  \omega_{1} \wedge \underline \omega_{2})'$ on $\hat W^-_x$  which agrees with $\underline  \omega_1 \wedge \underline \omega_2$ on $\partial \hat W^-_x,$ hence
as pointed out above  

\begin{equation}\label {DDD2}
\begin{aligned}
\int _{W^-_x}  \omega_1\wedge \omega_2=     
\int _{\hat W^-_x}  \underline \omega_1\wedge \underline \omega_2=&
\int _{\hat W^-_x}  \underline \omega'= \sum _{y\mid x>y}\int _{\hat W_{x,y}\setminus x} \underline \omega'_{y,1}\wedge \underline \omega'_{y,2}\\
=\sum _{y\mid x>y}\int _{\hat W_{x,y}\setminus x} \underline \omega'_{y,1}\wedge \underline \omega'_{y,2}=&
\sum _{y\mid x>y}\int _{W^-_y\times \mathcal M(x,y)} \underline \omega'_{y,1}\wedge \underline \omega'_{y,2}
\end{aligned}
\end{equation} 
 Since $ \underline \omega'_{y,1}$ and $ \underline  \omega'_{y,2}$ have separated variables one has 
\begin{equation} \label {DDD3}
\int _{W^-_y\times \mathcal M(x,y)}  \underline \omega'_{y,1} \wedge  \underline \omega'_{y,2}= \int  _{W^-_y}\omega_1    \cdot     \int_{\mathcal M(x,y)} \omega_2    
\end{equation}
Combining (\ref {DDD2}) and (\ref {DDD3}) one obtains (\ref {DD}).

\end{document}